\documentclass[11pt]{article}
\usepackage[T1]{fontenc}
\usepackage{mathtools}
\usepackage{amssymb}
\usepackage{amsthm}
\usepackage{mathptmx}
\usepackage{semantex}
\usepackage{xfrac}
\usepackage{tikz-cd}
\usepackage[margin=0.8in]{geometry}
\usepackage{blindtext}
\usepackage{tgcursor}
\usepackage{verbatim}
\usepackage{comment}
\usepackage{todonotes}
\usepackage{bbold}
\usepackage{titlesec}
\usepackage{authblk}

\usepackage[backend=biber,style=numeric]{biblatex}
\usepackage{csquotes}

\NewSymbolClass\MyBinaryOperator[
    define keys={
        {Lder}{command=\overset{L}},
        {Rder}{upper=R},
    },
]
\NewObject\MyBinaryOperator\tensor{\otimes}[
    define keys={
        {der}{Lder},
    },
]
\NewObject\MyBinaryOperator\fibre{\times}[
    define keys={
        {der}{Rder},
    },
]

\DeclareMathOperator{\Id}{Id}

\DeclareMathOperator{\Hom}{Hom}
\DeclareMathOperator{\Ext}{Ext}
\DeclareMathOperator{\Tor}{Tor}

\DeclareMathOperator{\lin}{lin}

\DeclareMathOperator{\Mod}{Mod}

\DeclareMathOperator{\Stab}{Stab}
\DeclareMathOperator{\Vect}{Vect}

\titleformat{\section}
  {\normalfont\scshape\filcenter}{\thesection}{1em}{}

\def\arr{\longrightarrow}
\def\dt{\text{.}}

\theoremstyle{plain}
\newtheorem{theorem}{Theorem}[section]
\newtheorem{lemma}[theorem]{Lemma}
\newtheorem{corollary}[theorem]{Corollary}
\newtheorem{df}[theorem]{Definition}
\newtheorem{remark}[theorem]{Remark}

\title{Bounded functor cohomology and comparison theorem}
\author{Karol Janowicz}
\affil{Institute of Mathematics, University of Warsaw, Banacha 2, 02-097 Warsaw, Poland}
\date{}

\begin{document}

\maketitle

\section*{Introduction}
In this work we present a possible approach to compute cohomology of linear groups over finite fields via functor cohomology. More precisely, we consider bounded polynomial functor categories (introduced by M. Chałupnik and P. Jaśniewski in \cite{chałupnik2022strict}) and prove the unstable version of the weak comparison theorem with the ordinary functor category. The inspiration for this result comes directly from the work of Franjou-Friedlander-Suslin-Scorichenko in \cite{franjou1999general} and from the work of Cline-Parshall-Scot-van der Kallen presented in \cite{Cline1977RationalAG}. The unstable character of the latter results was crucial in the proof.

In the first section we show that the cohomology of the general linear group over a finite field may be directly reformulated into the $\Ext$ algebra of the exterior power functor $\Lambda^n$. This is explicitly done by the extension by zero procedure that passes from representations of the general linear group to representations of the full monoid of matrices. This part is inspired by the work of Kuhn in \cite{Kuhn_1994}, where the extension by zero functor $k[GL_n]-\Mod \arr k[M_n] - \Mod$ appears as a part of a certain recollement of abelian categories. We prove that this recollement upgrades to the recollement of the corresponding derived categories and therefore there is a connection between cohomology of general linear group and the functor cohomology. In the second section we study the connections between the cohomology of ordinary functors and the cohomology of strict polynomial functors (in a Friedlander-Suslin sense, see \cite{Friedlander_1997}), all functor categories remaining bounded. In particular, we prove the unstable version of the weak comparison theorem proven in \cite{franjou1999general}. Here we vastly rely on the work of Cline-Parshall-Scot-van der Kallen in \cite{Cline1977RationalAG}.

\section{Extension by zero and general linear cohomology}
Let $k$ be a finite field and let $n>0$ be any positive integer. Let $M_n(k)$ denote a monoid of matrices $n \times n$ with coefficients $k$ and let $GL_n(k)$ be a general linear group over $k$ (we will omit the coefficients, if it is clear what they are). Denote by $\mathcal{F}_n$ the category of all functors from the category $\Vect_{\leq n}$ of vectors spaces of dimension $\leq n$ into the category of finite dimensional vector spaces $\Vect$ over $k$. We recall that by the Gabriel-Mitchell theorem (cf. \cite{bass1968algebraic}, II Theorem 1.3) we have an equivalence of categories $\mathcal{F}_n \simeq k[M_n]-\Mod$, where $k[M_n]$ denotes the semigroup ring of the monoid of matrices. The equivalence is given by the evaluation functor $F \mapsto F(k^n)$ and therefore we will sometimes implicitly identify a functor $F \in \mathcal{F}_n$ with the corresponding $k[M_n]$-module $F(k^n)$. Let $j:GL_n(k) \arr M_n(k)$ be the inclusion map and let $j_*:k[M_n]-\Mod \arr k[GL_n] - \Mod$ be the corresponding forgetful functor. Consider the ring map $i:k[M_n] \arr k[GL_n]$ mapping every non-singular matrix to itself and every singular matrix to $0$ (then extend linearly). Thus there is an induced functor $i_*:k[GL_n]-\Mod \arr k[M_n]-\Mod$ and $j_* \circ i_* = \Id$ ($i_*$ will be called an \textbf{extension by zero functor}). In particular, the map
\[    \Ext^*(A,B) \arr \Ext^*(i_*A, i_*B)     \]
is injective for any $A,B \in k[GL_n]-\Mod$.

In this section, we will construct a projective resolution of $i_*k[GL_n]$. Namely, the resolving complex will consist of submodules $k[M_n \cdot e]$, where $e \in M_n$ is an idempotent. Note that the module $k[M_n \cdot e] \subset k[M_n]$ is clearly projective, since it is a direct summand of a free module $k[M_n]$ (via a map $A \mapsto A \cdot e$ of $k[M_n]$-modules). 

Before moving on to the construction, we will introduce some notations. Let $L_1,\ldots, L_m$ be a full list of lines in $k^n$. Then we choose idempotents $e_{i_0,i_1, \ldots, i_k} \in M_n(k)$ whose kernel is exactly a subspace spanned by the lines $L_{i_0}, \ldots, L_{i_k}$. The choice can be made arbitrarily, since we only care about the left $k[M_n]$-module structure, not a $k[M_n]-k[GL_n]$-bimodule structure. 

Consider the projective module $k[M_ne_{i_0,i_1, \ldots, i_k}] \subset k[M_n]$. As a $k$-vector space, $k[M_ne_{i_0,i_1, \ldots, i_k}]$ is spanned by matrices vanishing on the span of lines $L_{i_0}, \ldots, L_{i_k}$. In particular, there are "restriction" maps
\[    k[M_ne_I] \arr k[M_ne_J]  \]
whenever $J \subset I$.

There is a short exact sequence
\[ 0 \arr k[M_n^{\text{sing}}] \arr k[M_n] \arr i_*k[GL_n] \arr 0 \]
where $k[M_n^{\text{sing}}]$ denotes the module spanned on singular matrices. We will resolve $k[M_n^{\text{sing}}]$ in a simplicial way, using the rank filtration. Define
\[  C_k:= \bigoplus_{i_0<i_1 < \ldots < i_k} k[M_ne_{i_0,i_1, \ldots, i_k}] \dt   \]
Then we define the differential
\[   \partial: C_k \arr C_{k-1}  \]
as follows. An element $(x_I)_I \in C_k$, where $I=(i_0 < i_1 < \ldots < i_k)$ are the increasing indices sequences, is mapped into
\[   \partial\big(x_{(i_0,\ldots, i_k)} \big) :=\sum_{j=0}^k (-1)^j x_{(i_0, \ldots, \widehat{i_j}, \ldots, i_k)} \dt  \]

Note that formally, an element $x_{(i_0, \ldots, \widehat{i_j}, \ldots, i_k)}$ in the right hand side of the above equation is the restriction of $x_{(i_0,\ldots, i_k)}$ 

A straightforward calculation shows, that $\partial \circ \partial =0$, so $C_\bullet$ is a complex.

\begin{theorem}
\label{resolution of GL_n}
For the complex $C_\bullet$ constructed above we have $H_n(C_\bullet)=0$ for $n>0$ and $H_0(C_\bullet)=k[M_n^{\text{sing}}]$.
\end{theorem}
\begin{proof}
Note that the above complex is of the form 
\[    C_k=\bigoplus_{i_0<i_1 < \ldots < i_k}  k[X_{i_0} \cap \ldots \cap X_{i_k}]  \]
where $X_i$ is a set of matrices vanishing on line $L_i$. In other words, it is of the form
\[    C_k=k\Big[\bigsqcup_{i_0<i_1 < \ldots < i_k} X_{i_0} \cap \ldots \cap X_{i_k}\Big] \dt  \]
Let us consider a semi-simplicial set $S$ defined as follows: 
for any singular matrix $A$, define a simplex $\Delta_A$, whose vertices are in a $1-1$ correspondence with lines $L_1, \ldots, L_k$ on which $A$ vanishes. Then take $S=\bigsqcup \Delta_A$, where the sum is taken over all singular matrices $A \in M_n$. Note that the simplicial homology of $S$ is defined by the complex $C_\bullet$, and thus it is acylic in positive degrees and $H_0(C_\bullet)=k[M_n^{\text{sing}}]$.
\end{proof}

It follows that the complex
\[   C_\bullet \arr k[M_n] \arr 0  \]
is a projective resolution of $k[GL_n]$ which enables to do some homological computations. As an example we note the following corollary.

\begin{corollary}
The $k[M_n]$-module $i_* k[GL_n]$ is not projective for $n>1$.
\end{corollary}
\begin{proof}
We compute $\Tor_1^{k[M_n]}(i_* k[GL_n],\Lambda^{n-1})$ using the constructed resolution of $i_* k[GL_n]$. Tensoring it with a functor $F \in \mathcal{F}_n$ we obtain the complex whose terms are of the form 
\[  C_s \otimes_{k[M_n]} F= \bigoplus_{|I|=s}  F(e_I k^n)   \]
and the differential induced by the inclusions. If $F=\Lambda^{n-1}$, then for all $s>1$ we have $C_s \otimes \Lambda^{n-1}=0$, therefore
\[   i_* k[GL_n] \tensor[k[M_n],der] \Lambda^{n-1}= \ldots \arr 0 \arr \bigoplus_{k^{n-1} \subset k^n} \Lambda^{n-1}(k^{n-1}) \arr \Lambda^{n-1} (k^n) \arr 0 \arr \ldots  \]
so $\displaystyle \Tor_1^{k[M_n]}(i_* k[GL_n],\Lambda^{n-1})=\ker \Big( \bigoplus_{k^{n-1} \subset k^n} \Lambda^{n-1}(k^{n-1}) \arr \Lambda^{n-1} (k^n) \Big) \neq 0$.
\end{proof}

\begin{remark}
For $n=1$, the module $i_* k[GL_1]=i_*k[k^*]$ is the direct summand of $k[M_1]=k[k_{add}]$, hence is projective.
\end{remark}

\begin{remark}
We remark that there is an analogous resolution of $k[GL_n]$ as a right $k[M_n]$-module with modules of the form $k[E_I M_n]$ where $E_I$ is a projection on the corresponding subspace in $k^n$. 
\end{remark}

\begin{remark}
In the proof we haven't used what the coefficients of the representations are, so in the same we may construct the projective resolution of $R[M_n]$-module $i_*R[GL_n]$ for any commutative ring $R$.
\end{remark}

One of the most important corollaries of the above construction is the following result on the extension by zero functor $i_*$.

\begin{theorem}
\label{Kuhn's recollement on derived level}
Let $A,B$ be any $k[GL_n]$-modules. Then the map
\[   i_*: \Ext_{k[GL_n]}^*(A,B) \arr \Ext_{k[M_n]}^*(i_*A,i_*B)    \]
is an isomorphism. Therefore, there is a recollement of the unbounded derived categories
\[    \begin{tikzcd}
{\mathcal{D}\Big(k[GL_n]\Big)} \arrow[rr, "i_*"] &  & {\mathcal{D}\Big(k[M_n]\Big)} \arrow[rr, "j_n^*"] \arrow[ll, "i^!", bend left] \arrow[ll, "i^*"', bend right] &  & {\mathcal{D}\Big(k[M_{n-1}]\Big)} \arrow[ll, "j_*", bend left] \arrow[ll, "j_!"', bend right] \dt
\end{tikzcd}  \]
where $j_n:k[M_{n-1}] \arr k[M_n]$ denotes the standard inclusion $A \mapsto \left[\begin{array}{ c | c }
    A & 0 \\
    \hline
    0 & 0
  \end{array}\right]$.
\end{theorem}
\begin{proof}

First we prove that
\[   \Ext^*_{k[M_n]}(k[GL_n],k[GL_n])=0  \]
for all $* >0$. Note that
\[    \Hom_{k[M_n]}(k[M_ne] , k[GL_n])=0   \]
since any $k[M_n]$-linear map $\varphi:k[M_ne] \arr k[GL_n]$ satisfies $\varphi(e)=\varphi(e^2)=e \cdot \varphi(e)=0$. Thus the complex $\Hom_{k[M_n]}(C_\bullet, k[GL_n])$ is acyclic, which implies the vanishing of cohomology $\Ext_{k[M_n]}^*(k[GL_n],k[GL_n])$ for $*>0$. So the result on $\Ext$ groups follows directly from the Theorem 4.4. in \cite{GEIGLE1991273}

The recollement of the derived categories follows immediately from the first part of the theorem and the recollement of the underlying abelian categories proved in \cite{Kuhn_1994}, Theorem 2.3.
\end{proof}

Let $\det$ be a determinant representation of $GL_n(k)$. Note that $i_* \det$ is isomorphic to the evaluation of the exterior power functor $\Lambda^n$ on the $n$-dimensional space $k^n$. Therefore, Theorem \ref{Kuhn's recollement on derived level}  together with an equivalence of categories $\mathcal{F}_n \simeq k[M_n]-\Mod$ gives us the following corollary on cohomology of general linear groups.
\begin{corollary}
Let $k$ be a finite field. Then
\[    H^*( GL_n(k),k) \simeq \Ext^*_{\mathcal{F}_n}(\Lambda^n,\Lambda^n) \dt \]
\end{corollary}       

\begin{proof}
We have
\[    H^*(GL_n(k),k) \simeq \Ext^*_{GL_n(k)}(k,k) \simeq \Ext^*_{GL_n(k)}(\det,\det) \simeq \Ext^*_{\mathcal{F}_n}(\Lambda^n,\Lambda^n)    \]
where the last isomorphism is induced by the functor $i_*$.
\end{proof}

We remark that the above result is especially meaningful for $k=\mathbb{F}_p$ and small $n$. This result motivates us to understand the cohomology of the exterior power functor $\Lambda^n$ using the fact that it is the polynomial functor. Here comes the idea of comparing the ordinary functor category $\mathcal{F}_n$ with the category of polynomial functors $\mathcal{P}_{d,n}$, which was done in the stable case in \cite{Franjou_2008}. In the next section we compare these categories provided that the field is bigger than the degree of the functor.

\section{Comparison theorem in bounded functor categories}

In this part we will prove the unstable version of the weak comparison theorem appearing in \cite{franjou1999general} (Theorem 2.7). By unstable we mean the setup of bounded functor categories, i.e. the functors defined on vectors spaces of dimension $\leq n$.

First we recall the comparison theorem on rational group cohomology (cf. \cite{Cline1977RationalAG}, Theorem 6.6). Let $k$ be a finite field and $M,N$ be finite dimensional, rational $GL_n$-modules. Here $GL_n$ denotes the algebraic general linear group and $GL_n(k)$ - its group of $k$-points. Then for any positive integer $m$, there is a finite field extension $k \subset K$ and big enough $i$ such that the natural map
\[    \Ext^*_{GL_n}(M^{(i)},N^{(i)})  \arr \Ext^*_{GL_n(K)}(M,N) \]
is an isomorphism for all $0 \leq * \leq m$. Here $M^{(i)}$ denotes the $i$-th Frobenius twist of $M$. Note also, that the LHS stabilizes with respect to $i$, the number of Frobenius twists [ref. needed].
 
Let $\mathcal{P}_{d,n}$ denote the category of strict polynomial functors of degree $d$ defined on the category of $k$-vector spaces of dimension $\leq n$. Recall that the category $\mathcal{P}_{d,n}$ is equivalent to the category of $\Gamma^d(End(k^n))$-modules, as shown in \cite{chałupnik2022strict}. It is known that this category is equivalent to the full subcategory of the category of rational $GL_n$-modules consisting of polynomial, rational $GL_n$-modules of degree $d$. Moreover, it can be shown that the inclusion $Pol_{n,d} \arr  GL_n-\Mod$ induces an isomorphism on $\Ext$ groups (see \cite{Friedlander_1997}). It follows that the following composition
\[  \begin{tikzcd}
{\Ext^*_{\mathcal{P}_{dp^i,n}}(F^{(i)},G^{(i)})} \arrow[r] & {\Ext^*_{GL_n}(F^{(i)},G^{(i)})} \arrow[r] & {\Ext^*_{GL_n(K)}(F(K^n),G(K^n))}
\end{tikzcd}  \] 
is an isomorphism for $* \leq s$ for big enough field extension $K$ and big enough $i$. This isomorphism can be written as the composition
\[ \begin{tikzcd}  {\Ext^*_{\mathcal{P}_{dp^i,n}}(F^{(i)},G^{(i)})} \arrow[r] & {\Ext^*_{\mathcal{F}_n}(F,G)} \arrow[r] & {\Ext^*_{GL_n(K)}(F(K^n),G(K^n))} \dt
\end{tikzcd}    \]

As a corollary, the restriction map
\[  \begin{tikzcd} 
{\Ext^*_{\mathcal{F}_n}(F,G)} \arrow[r] & {\Ext^*_{GL_n(K)}(F(K^n),G(K^n))}
\end{tikzcd}   \]
is an epimorphism in a range (over a sufficiently big field extension $K$). We claim that both of these maps are actually isomorphisms.
\begin{theorem}[weak comparison theorem for bounded functor categories]
\label{weak comparison theorem for bounded functor categories}
Let $F,G \in \mathcal{P}_{d,n}$ be polynomial functors over a finite field $k$ and let $s \geq 0$ be an integer. Then for big enough field extension $k \subset K$ and big enough $i$, both depending on $s$ and the degree $d$, the restriction map
\[    \Ext^*_{\mathcal{P}_{dp^i,n}}(F^{(i)},G^{(i)}) \arr \Ext^*_{\mathcal{F}_n}(F,G)  \]
is an isomorphism for all $0 \leq * \leq s$. Hence, for the same field extension $K$, the restriction map
\[  \Ext^*_{\mathcal{F}_n}(F,G) \arr   \Ext^*_{GL_n(K)}(F(K^n),G(K^n)) \]
is an isomorphism for all $0 \leq * \leq s$.
\end{theorem}
 \begin{proof}

 As noticed earlier, it is enough to show that for big enough $K$, the restriction map
\[  \begin{tikzcd} 
{\Ext^*_{\mathcal{F}_n}(F,G)} \arrow[r] & {\Ext^*_{GL_n(K)}(F(K^n),G(K^n))}
\end{tikzcd}   \]
is an isomorphism in the given range. Choose an extension $k \subset K$ so that $|K| \geq d$ and for all polynomial functors $f,g \in \mathcal{P}_{d,n}$ the map
\[
\begin{tikzcd}
\Hom{\mathcal{P}_{dp^i,n}}(f^{(e)},g^{(e)}) \arrow[r] &  {\Hom_{GL_n(K)}(f(K^n),g(K^n))}
\end{tikzcd} 
\]
is an isomorphism for $e \gg 0$. This can be done uniformly by \cite{Cline1977RationalAG}, Theorem 6.6. Indeed, any functor $f \in \mathcal{P}_{d,n}$ induces a polynomial, rational $GL_n$-module of degree $d$ (see Theorem 2.2 in \cite{chałupnik2022strict}), so there are finitely many possible weights of the maximal torus $T \in GL_n$.  Then the map
\[ \Hom_{\mathcal{F}_n}(f,g) \arr \Hom_{GL_n(K)}(f(K^n),g(K^n))    \]
is an epimorphism, hence also an isomorphism. This amounts to the fact that for such $K$, the restriction map \[ j_*:K[M_n(K)]-\Mod \arr K[GL_n(K)]-\Mod \] is fully faithfull while restricted to the full subcategory $\mathcal{P}_{d,n} \subset K[M_n(K)]-\Mod$ (note that the size of $K$ depends only on $d$). In particular, we can treat $\mathcal{P}_{d,n}$ as a full subcategory of $GL_n(k)-\Mod$. Thus it would be sufficient to prove that for any strict polynomial functor $F$ we can find such a big extension $K$, so that
\[   \Ext^i_{GL_n(K)}(K[M_n],F(K^n))=0   \]
for all $s \geq i >0$, as the following lemma says.
\begin{lemma}
Let $\mathcal{A},\mathcal{B}$ be abelian categories, where $\mathcal{B}$ admits enough injectives. Let
\[
\begin{tikzcd}
\mathcal{A} \arrow[r, "F"] & \mathcal{B} \arrow[l, "G", bend left]
\end{tikzcd}
\]
be a pair of adjoint functors, $F \dashv G$. Assume that $F$ is exact and fully faithful on a subcategory $\mathcal{A}' \subset \mathcal{A}$ and that $R^i G(FA')=0$ for all $A' \in \mathcal{A}'$ and $s \geq i>0$. Then for any $A',B' \in \mathcal{A}'$, the map
\[  \begin{tikzcd} 
{\Ext^i_{\mathcal{A}}(A',B')} \arrow[r] & {\Ext^i_{\mathcal{B}}(F(A'),F(B'))}
\end{tikzcd}   \]
is an isomorphism for all $s \geq i \geq 0$.
\end{lemma}
\begin{proof}
Let $FB' \arr I_\bullet$ be an injective resolution of $FB'$. Since $R^i G(FB')=0$ for all $0<i \leq s$ and $G$ preserves injective objects, the complex $G(I_\bullet)_{\bullet \leq s+1}$ may be extended an injective resolution of an object $GFB'$, which is isomorphic to $B'$ via the counit map. Therefore
\[    \Ext^i_{\mathcal{A}}(A',B')\simeq H^i\big(\Hom_{\mathcal{A}}(A',G(I_\bullet))\big) \simeq H^i(\Hom_{\mathcal{B}}(FA',I_\bullet)) \simeq \Ext^i_{\mathcal{B}}(FA',FB') \dt  \]
for all $0 \leq i \leq s$.
\end{proof}

Therefore we need to study $K[M_n]$ treated as a representation of $GL_n$ (via left multiplication on $M_n$). Consider the split rank-filtration of $k[M_n]$
\[  k[M_n] \simeq \bigoplus_{r=0}^n k[M_n^r]  \]
where $M_n^r \subset M_n$ denotes the subset of all matrices of rank $r$. Note that $A,B \in M_n^r$ are in the same orbit under $GL_n$-action if and only if $\ker A=\ker B$. In fact there is no distinction between different orbits within given rank, i.e. $GL_n \cdot A \simeq GL_n \cdot B$ via equivariant map, provided that $A,B \in M_n^r$. For this, let $C \in GL_n$ be such that $C(\ker B)=\ker A$ and consider the map 
\[  \varphi_C:GL_n \cdot A \arr GL_n \cdot B , \hspace{30pt} GA \mapsto GAC \dt\]
It is well-defined (since $\ker(GAC)=\ker B$, hence $GAC \in GL_n \cdot B$) and of course $GL_n$-equivariant. Therefore it is sufficient to establish vanishing of higher derived functors $R^i\Hom_{k[GL_n]}(k[GL_n \cdot I_r],-)$, where $I_r$ denotes the standard projection matrix on $\lin(e_1, \ldots, e_r)$,
\[   I_r=  \left[\begin{array}{ c | c }
    \Id_{r \times r} & 0 \\
    \hline
    0 & 0
  \end{array}\right] \dt \]

We compute that
\[   GL_n^r:= \Stab(I_r)=\Big\{ \left[\begin{array}{ c | c }
    \Id_{r \times r} & B \\
    \hline
    0 & A
  \end{array}\right] \in GL_n |\hspace{10pt} A \in GL_{n-r},B \in M_{r \times(n-r)} \Big\} \]
and therefore
\[    \Ext^i_{k[GL_n]}(k[GL_n \cdot I_r],-) \simeq H^i(GL_n^r,-) \dt   \]
Note that the subgroup of matrices of the form $\left[\begin{array}{ c | c }
    \Id_{r \times r} & A \\
    \hline
    0 & \Id_{(n-r) \times (n-r)}
  \end{array}\right] \subset GL_n^r$, isomorphic to the additive group $K_{add}^{r(n-r)}$, is a normal subgroup of $GL_n^r$, so $GL_n^r \simeq K_{add}^{r(n-r)} \rtimes GL_{n-r}$. Hence, by the Lyndon–Hochschild–Serre spectral sequence, it is enough to prove that for big enough field extensions $k \subset K$, we have
  \[   H^p(GL_{n-r},H^q( K_{add}^{r(n-r)},F)) =0   \]
  for $p+q>0$.
  
   The strategy to show it is following; we will prove that for every $r,i$, the (bounded) functor
   \[      V \mapsto H^i(\Hom(V,k^r),F(k^r \oplus V) ) \]
   is of degree $\leq d$ (in the terms of the action of the center $k^* \subset GL_n(k)$), so that the following vanishing result will do the job.
   
   \begin{df}
   Let $F \in \mathcal{F}_n$ be a functor over $k$. We will say that $F$ is of degree $d$, if $d$ is the only weight of the action of the center $k^*\subset GL_n(k)$, so
   \[    F(\lambda \cdot )(\alpha)=\lambda^d \cdot F(\alpha)   \]
   for all $\alpha \in k^n$ and $\lambda \in k^*$. The full subcategory of functors of degree $d$ will be denoted by $\mathcal{F}_{d,n} \subset \mathcal{F}_n$.
   \end{df}
   
   \begin{lemma}
   \label{vanishing of cohomology in coeffecients in a functor}
   Let $F \in \mathcal{F}_n$ be a functor of degree $d$, that is every scalar matrix $\lambda \in GL_n$ acts on $F(k^n)$ via multiplication by $\lambda^d$. Let $k \subset K$ be a field extension with $|K| \geq d$. Then
   \[    H^*(GL_n(K),F(K^n))=0 \dt    \]
  In particular, for every $F \in \mathcal{P}_{d,n}$ we have
  \[    H^*(GL_n(K),F(K^n))=0 \dt    \]
   \end{lemma}
   \begin{proof}
   Let $\lambda \in GL_n(K)$ be the scalar matrix corresponding to $\lambda^d \neq 0,1$ in $K$. Then $\lambda$ acts on $F(k^n)$ by multiplication by $\lambda^d \neq 1$, so it acts on cohomology $H^*(GL_n(K),F(K^n))$ by multiplication by $\lambda^d$ (since $\lambda$ is central in $GL_n(K)$). On the other hand, any element $g \in GL_n(K)$ acts trivially on cohomology $H^*(GL_n(K),F(K^n))$, hence $H^*(GL_n(K),F(K^n))=0$.
   \end{proof}
   
   We will need one more property of functors of degree $d$.

\begin{lemma}
\label{degree is preserved under extensions}
Let 
\[  0 \arr F \arr G \arr H \arr 0   \]
be a short exact sequence of functors. Then $F,H$ are of degree $\leq d$ if and only if $G$ is of degree $\leq d$. In particular, if a (finite) functor is filtered by the functors of degree $\leq d$, then it is also of degree $\leq d$.
\end{lemma}
\begin{proof}
Restrict the given sequence to the sequence of representations of the center $k^* \subset GL_n(k)$. Since every such a representation is semisimple, we know that
\[     G(k^n) \simeq F(k^n) \oplus H(k^n)   \]
as representations of $k^*$ and therefore every weight in $G(k^n)$ comes from $F(k^n)$ or $H(k^n)$. The induction on the number of components in filtration shows that filtration of degree $\leq d$ implies degree $\leq d$.
\end{proof}

Define the following functor $F_r^i \in \mathcal{F}_{n-r}$:
  \[   F^i_r(V):=H^i(\Hom(V,k^r),F(k^r \oplus V) )  \]
  where $\Hom(V,k^r)$ is the additive group of linear maps, and $F(k^r \oplus V)$ is given the structure of representation of $\Hom(V,k^r)$ via action of matrices of the form $\left[\begin{array}{ c | c }
    I_{r \times r} & A \\
    \hline
    0 & I_{(n-r) \times (n-r)}
  \end{array}\right]$. Indeed, note that for any map $\varphi:V \arr W$ we obtain a map of $\Hom(W,k)$-representations $F(\Id \oplus \varphi):F(k^r \oplus V) \arr F(k^r \oplus W)$ (as quick computations $(\alpha,\beta) \mapsto (\alpha,\varphi(\beta)) \mapsto \alpha+A\varphi(\beta),\varphi(\beta))$ and $(\alpha,\beta) \mapsto (\alpha+A \varphi(\beta),\beta) \mapsto (\alpha+A \varphi(\beta),\varphi(\beta))$ show). Therefore there is also the induced map on cohomology 
  \[  H^*(\Hom(V,k^r),F(k^r \oplus V) ) \arr  H^*(\Hom(W,k^r),F(k^r \oplus W) ) \dt \]
  This is functorial in $V$, since the restriction map on the group cohomology is functorial. We would like to show that in a sufficiently big field extension $K$, $F^i_r$ is of degree $\leq d$, i.e. $F^i_r \in \bigoplus_{e=0}^d \mathcal{F}_{e,n-r}$ (here $\mathcal{F}_{d,n} \subset \mathcal{F}_n$ denotes the subcategory of functors of degree $d$).
  
 For $i=1$ we simply get the following map on morphisms-sets
 \[   \Hom\big(\Hom(V,k^r),F(k^r \oplus V) \big) \arr \Hom\big(\Hom(W,k^r),F(k^r \oplus W)\big)   \]
 so that $f \mapsto F(\Id \oplus \varphi) \circ f \circ \varphi^*$ for any $\varphi :V \arr W$. Notice that that the functor $\tilde{F}_r : V \mapsto F(k^r \oplus V)$, $\tilde{F}_r:\varphi \mapsto F(\Id \oplus \varphi)$ is again polynomial of degree at most $d$, which simply follows from the fact that the cross-effects of the strict polynomial functor of degree $d$ is again strict polynomial of degree $< d$.
 
 For $i>1$ we will use the degree-preserving filtration of the cohomology of additive group in coefficients $F$ by the cohomology in trivial coefficients, which is well-known to be isomorphic to
 \[   H^*(K_{add},K) \simeq \Lambda(V) \otimes S(W)    \]
 where $V$ is spanned on $m$ vectors (where $|K|=p^m$) in the cohomological degree $1$ and $W$ is the Bockstein of $V$ (see Theorem 4.1 in \cite{Cline1977RationalAG}). In particular, $H^1(K_{add},K) \simeq V$ and therefore, if $F$ happens to have a trivial evaluation as a representation of the additive group $\Hom(V,K^r)$, then the result follows from the above description and the Künneth formula.

 If $F$ has a non-trivial evaluation, then we resolve $F$ with the tensors of symmetric power functors $S_\bullet$ (which are known to co-generate $\mathcal{P}_{d,n}$) so there is a hypercohomology spectral sequence
 \[     H^i\big(K_{add},S_j(K^r \oplus V)\big) \implies H^{i+j}\big(K_{add},F(K^r \oplus V)\big) \dt \]
We make the following observation: the functor $V \mapsto S^d(K^r \oplus V)$ is filtered by the following subrepresentations of the group $K_{add} \rtimes K^* \subset GL_n$:
 \[    S^i(K^r)\otimes S^{d-i}(V)  \]
starting from $i=0$ to $i=d$. This observation implies that every representation $S_j(K^r \oplus V)$ is filtered with trivial $K_{add}$-modules with weight $\leq d$ of the maximal torus $K^* \subset GL(V)$. Thus, the group cohomology 
 \[ H^i\big(K_{add},S_j(K^r \oplus V)\big) \]
 is also filtered with the subquotients of $H^i(K_{add},K)$ and therefore by Lemma \ref{degree is preserved under extensions} it has also degree $\leq id$ (we need to multiply by $i$, since in the $i$-th cohomological degree I do have a composition of the given functor $F$ and polynomial functors of degree $\leq i$), so by lemma \ref{vanishing of cohomology in coeffecients in a functor} we can find such a big field extension $K$, for which $H^i(GL_n^r,F)=0$ for all $r$ and therefore $\Ext^i_{GL_n(K)}(K[M_n],F(K^n))=0$ for all $0 < i \leq s$, what we wanted to prove.
\end{proof}

We remark that in the above argument we show that the vanishing of the group cohomology $H^*(GL_n^r,F)$ in a sufficiently big field extension $K$ for all $r=0, \ldots, n$ and for all polynomial functors of degree $d$. It turns out (see Theorem \ref{vanishing of cohomology in coeffecients in a functor}) that the only relevant subgroup is the group of matrices
\[ \Big\{ \left[\begin{array}{ c | c }
    \Id_{r \times r} & A \\
    \hline
    0 & \lambda \cdot \Id
  \end{array}\right] |\hspace{10pt} A  \in M_{r \times(n-r)}, \lambda \in K^* \Big\}     \subset GL_n^r \]
  which is isomorphic to the semi-direct product $G:=\Hom(V,k^r) \rtimes K^*$. Here $K^*$ corresponds to the center of $GL(V)$.  The Clifford theory (cf. \cite{nagao1989representations}, 2. Theorem 3.1) asserts that every irreducible representation of this subgroup is $1$-dimensional, hence we can find the appropriate filtration of $F$ for free.
  \begin{lemma}
  Let $V$ be an irreducible $K$-representation of the group $G=K_{add}^{m} \rtimes K^*$. Then $V$ is $1$-dimensional.
  \end{lemma}
  \begin{proof}
  Note that the only irreducible representation of the group $N:=K_{add}^m$ is the trivial representation. Let $W \subset V$ be the trivial $N$-subrepresentation of $V$. Then $\sum_{g \in G} gW \subset V$ is a $G$-subrepresentation of $V$, hence $\sum_{g \in G} gW =V$ by irreducibility. Note that $N$ acts on each part $gW$ trivially, hence $N$ acts on $V$ trivially as well. Since every irreducible representation of $K^*$ is $1$-dimensional, we are done.
  \end{proof}

\printbibliography

\end{document}